\documentclass[11pt]{amsart}
\usepackage{amsmath,amssymb,amsthm}
\usepackage{tikz-cd}
\usepackage[margin=1in]{geometry}
\usepackage{amssymb,amsthm}
\usepackage{mathtools}
\usepackage{eucal,stmaryrd}
\usepackage{enumerate}
\usepackage[margin=1in]{geometry}
\usepackage{quiver}
\usepackage{hyperref}
\usepackage{xcolor}
\hypersetup{
    colorlinks,
    linkcolor={red!50!black},
    citecolor={blue!50!black},
    urlcolor={blue!80!black}
}
\theoremstyle{plain}
\newtheorem{theorem}{Theorem}[section]
\newtheorem{proposition}[theorem]{Proposition}
\newtheorem{lemma}[theorem]{Lemma}
\newtheorem{corollary}[theorem]{Corollary}

\theoremstyle{definition}
\newtheorem{definition}[theorem]{Definition}
\newtheorem{example}[theorem]{Example}
\newtheorem{remark}[theorem]{Remark}

\newcommand{\ann}{{\scriptscriptstyle\perp}}

\DeclareMathOperator{\supp}{supp}
\DeclareMathOperator{\argmax}{argmax}
\DeclareMathOperator{\Cay}{Cay}
\DeclareMathOperator{\CD}{CD}
\DeclareMathOperator{\mix}{mix}

\DeclareMathOperator{\diag}{diag}

\begin{document}
\title[An estimate for positive definite functions]{An estimate for positive definite functions on finite abelian groups and its applications}
\author{Lixia Wang}
\address{State Key Laboratory of Mathematical Sciences, Academy of Mathematics and Systems Science, Chinese Academy of Sciences}
\email{wanglixia@amss.ac.cn}
\author{Ke Ye}
\address{State Key Laboratory of Mathematical Sciences, Academy of Mathematics and Systems Science, Chinese Academy of Sciences}
\email{keyk@amss.ac.cn}

\begin{abstract}
This paper concentrates on positive definite functions on finite abelian groups, which are central to harmonic analysis and related fields.  By leveraging the group structure and employing Fourier analysis,  we establish a lower bound for the second largest value of positive definite functions.  For illustrative purposes,  we present three applications of our lower bound: (a) We obtain both lower and upper bounds for arbitrary functions on finite abelian groups; (b) We derive lower bounds for the relaxation and mixing times of random walks on finite abelian groups.  Notably,  our bound for the relaxation time achieves a quadratic improvement over the previously known one; (c) We determine a new lower bound for the size of the sumset of two subsets of finite abelian groups. 
\end{abstract}

\maketitle

\section{Introduction}
\subsection*{Backgrouand} 
By endowing a finite set with an abelian group structure,  any problem on this set can be reformulated as one involving functions on a finite abelian group.  This perspective enables the introduction of various new techniques to study discrete problems.  For instance,  the technique of sum-of-squares on finite abelian groups has been developed to solve combinatorial optimization problems \cite{FSP16,BGP16,SL23,YYZ24}; When a graph is a Cayley graph,  its graph-theoretic properties are closely related to the representation-theoretic properties of the underlying group \cite{L.Babai,Friedman93,KKRT16,joel}; In the study of random walks,  the group structure of the state space is crucial for estimating the relaxation time and mixing time \cite{Greenhalgh89,Saloff04,LPW09,Hough17}.  Among all functions on a given abelian group,  positive definite functions constitute one of the most important classes.  According to the celebrated Bochner's Theorem,  a function is positive definite if and only if it is the Fourier transform of a non-negative measure.  Over the past half-century,  these functions have played a pivotal role in diverse fields,  including harmonic analysis \cite{Stewart76,  HV11,DS13, Banaszczyk22,Erb22},  probability theory \cite{Levy67, Price70,Heyer77} and quantum physics \cite{Lieb85, Caola91,Kovchegov99}.  

Let $G$ be a finite abelian group and let $f:G \to \mathbb{R}$ be a positive definite function.  In many scenarios \cite{L.Babai,Greenhalgh89, Friedman93,Saloff04,LPW09,KKRT16,joel,Hough17},  the second largest value of $f$:
\[
\nu_2(f) \coloneqq \max\{f(x):x\in G,\; f(x) < f(1)\},
\] 
is the quantity that lies at the core of the problem.  It can be readily shown that $\max_{x\in G} f(x) = f(1)$, where $1\in G$ is the identity element.  However,  determining $\nu_2(f)$ is generally much more challenging,  if not impossible.  The primary focus of this paper is to establish a lower bound for $\nu_2(f)$ and present three applications of this bound.
\subsection*{Main contributions} 
The main result of this paper is Theorem~\ref{cor2.3},  which states that
\begin{equation}\label{eq:main}
\nu_2(f) \ge f(1) \left( 1 - \frac{ \pi^2  |G|  }{2m^2 \left( |G|-2^t(m-1)^{\frac{s-t}{2}} \right)}  \right),
\end{equation}
for any positive integer $m < \left(2^{-t} |G| \right)^{1/r}+1$.  Here $s \coloneqq |\supp(\widehat{f})|$,  $t \coloneqq |\{\chi\in \supp(\widehat{f}): \chi^2 = 1\}|$ and $\supp(\widehat{f})$ is the support of the Fourier transform $\widehat{f}$ of $f$.  Consequently,  if we fix $s$ while allowing $G$ to vary,  then we obtain $\left( f(1) - \nu_2(f) \right)/f(1)  = O(|G|^{-4/s})$.  We illustrate three applications of \eqref{eq:main},  which are summarized as follows.
\begin{enumerate}[(a)]
\item Function value estimation: We obtain both lower and upper bounds of values of arbitrary functions on finite abelian groups (cf.  Proposition~\ref{prop:function value}).
\item Relaxation and mixing time estimation: We establish lower bounds for the relaxation and mixing times of random walks on finite abelian groups (cf.  Propositions~\ref{relaxation time} and \ref{mixing time}).  It is worth remarking that our lower bound for the relaxation time exhibits a quadratic improvement over the existing one \cite{Greenhalgh89,Hough17,MR23}.
\item Sumset size estimation: We derive a new type of lower bound for the size of the sumset of two subsets of finite abelian groups (cf.  Proposition~\ref{lem:sumset}).
\end{enumerate}
\subsection*{Organization of the paper}
The remainder of this paper is organized as follows.  For the reader's convenience,   we provide a brief review in Section~\ref{sec:pre} on Fourier analysis on finite abelian groups and weighted graphs.  In Section~\ref{sec:pd},  we define positive definite functions and establish their fundamental properties.  In Section~\ref{sec:nu2}, we extend the covering argument from graph theory to derive our lower bound \eqref{eq:main}. Additionally, for independent interest, we interpret our result from a graph-theoretic perspective.  Finally,  we present in Section~\ref{sec:app} the three aforementioned applications of our lower bound.
\section{Preliminaries}\label{sec:pre}
\subsection{Fourier Analysis on Finite Abelian Groups}
This subsection provides a brief overview of the Fourier analysis on finite abelian groups.  For more details,  interested readers are referred to \cite{Rudin62,FH91,Terras99,Steinberg12, TV10}.  Let $ G $ be a finite abelian group.  Throughout this paper, unless otherwise stated,  we assume  that the group operation on $G$ is multiplicative and denoted by $(x,y) \mapsto xy$.  

A \emph{character} of $G$ is a group homomorphism $ \chi: G \to \mathbb{C}^{\times}$,  where $\mathbb{C}^{\times} \coloneqq \mathbb{C} \setminus \{ 0 \}$ is the multiplicative group of $\mathbb{C}$.  The \emph{dual group} $\widehat{G}$ of $G$ consists of all characters on $G$,  whose group operation is the pointwise product of functions.  We denote by $L(G)$ the vector space of all $\mathbb{C}$-valued functions on $ G $,  equipped with the complex-valued inner product:
\[
\langle f, g \rangle = \frac{1}{|G|} \sum_{x \in G} f(x) \overline{g(x)},\quad f,g\in L(G).
\]
The \emph{Fourier transform} of $f$ is the function $\widehat{f}: \widehat{G} \to \mathbb{C}$ defined by $\widehat{f}(\chi) \coloneqq \langle f,  \chi \rangle$.  For each $f\in L(G)$,  we have the \emph{Fourier expansion} $f = \sum_{\chi\in \widehat{G}} \widehat{f}(\chi) \chi$,  where $\widehat{f}(\chi)$ is called a \emph{Fourier coefficient} of $f$.  Given $f,g\in L(G)$,  the \emph{convolution} of $f$ and $g$ is a function $f * g \in L(G)$ defined by
\[
(f * g)(x) = \frac{1}{|G|} \sum_{y \in G} f(xy^{-1})g(y),\quad x \in G.
\]
The basic properties of the Fourier transform and convolution are summarized below. 
\begin{theorem}[\cite{Nathanson00}]
For any $f,g\in L(G)$,  we have 
\begin{enumerate}[(a)]
\item $\widehat{(f * g)} = \widehat{f}  \widehat{g}.
$
\item Parseval identity: $\langle f, f \rangle = |G| \langle \widehat{f}, \widehat{f} \rangle = \sum_{\chi \in \widehat{G}} |\widehat{f}(\chi)|^2$.	
\item $\widehat{\widehat{f\,}}(x) = \frac{1}{|G|} f(x^{-1})$.
\end{enumerate}
\end{theorem}

For each subgroup $H$ of $ G $,  we define
\[
H^{\ann} \coloneqq \{\chi \in \widehat{G}: \chi(x) = 1,\; x \in H\}.
\]
Since there is a natural isomorphism between $ \widehat{\widehat{G\,}} $ and $ G $,  we have $ (H^{\ann})^{\ann} = H$ under this identification.  Moreover,  the subgroups of $G$ and $\widehat{G}$ are in a one-to-one correspondence.
\begin{lemma}\cite[Problem~2.7]{I.M}
\label{Lemma 1.1}
The assignment $H \mapsto H^\ann$ is a bijection between the set of subgroups of $G$ and the set of subgroups of $\widehat{G}$.  Moreover,  $H_1 \subseteq H_2$ if and only if $H_2^\ann \subseteq H_1^\ann$. 
\end{lemma}

We denote by $ \pi: G \to G/H $ the quotient map and let $\pi^\ast: L(G/H) \to L(G)$ be the map defined by $ \pi^\ast(f) = f \circ \pi$.  
\begin{lemma}[\cite{Nathanson00}, Lemma 4.5]
\label{Lemma 1.2}
We have $\pi^*\left( \widehat{G/H}\right) \subseteq H^{\ann}$ and the restriction map $\pi^*: \widehat{G/H} \to H^{\ann}$ is an isomorphism of groups.  
\end{lemma}
\subsection{Weighted graphs}\label{subsec:analysis on graph}
The primary goal of this subsection is to introduce weighted graphs and record an inequality involving their eigenvalues.  Standard references are \cite{LP18,LY10,KKRT16}.  Let $ \mathcal{G} = (V, E, w) $ be an undirected, weighted and connected graph.  Here $ V $ is the set of vertices,  $E$ is the set of edges and $\omega: V \times V \rightarrow \mathbb{R}_{\geq 0} $ satisfies $\omega(x,y) = \omega(y,x)$ for any $(x,y)\in V \times V$,  and $\omega(x,y) > 0$ if and only if $\{x,y\} \in E$.  The function $\omega$ is called the \emph{weight function} of $\mathcal{G}$.  Given $x,y\in V$,  we say that $x$ and $y$ are \emph{adjacent},  denoted as $x \sim y$,  if $\{x,y\} \in E$.  The \emph{degree of a vertex} $ x \in V $ is  $ d_x \coloneqq \sum_{y \in V,\; y \sim x} \omega(x,y) $.  We denote $d_G \coloneqq \max_{x \in V} d_x$.

Let $L(V,\mathbb{R})$ be the vector space consisting of all real-valued functions on $V$.   The \emph{Laplacian operator} of $\mathcal{G}$ is the linear map $\Delta: L(V,\mathbb{R}) \to L(V,\mathbb{R})$ defined by
\begin{equation}\label{eq:laplace}
\left( \Delta u \right) (x) = \sum_{\substack{y \in V \\ y \sim x}} \omega(x,y) (u(y) - u(x)),\quad u\in L(V,\mathbb{R}),\;  x \in G.
\end{equation}
We remark that $\Delta = A - D$ where $A \coloneqq (\omega(x,y))_{x,y\in V}\in \mathbb{R}^{|V| \times |V|}$ and $D \coloneqq \diag(d_x)_{x\in V} \in \mathbb{R}^{|V| \times |V|}$.

Moreover,  we define two bilinear operators on $L(V,\mathbb{R})$:
\begin{equation}\label{eq:Gamma}
\Gamma(u, v) = \frac{1}{2} \left( \Delta(u v) - u\Delta v - (\Delta u)v \right),  \quad 
\Gamma_2(u, v) = \frac{1}{2} \left( \Delta \Gamma(u, v) - \Gamma(u,  \Delta v) - \Gamma(\Delta u, v) \right),
\end{equation}
where $ u, v \in L(V,\mathbb{R})$.  For simplicity,  we denote $\Gamma(u) \coloneqq \Gamma(u,u)$ and $\Gamma_2(u) \coloneqq \Gamma_2(u,u)$.  We say that $\mathcal{G}$ is \emph{$\CD(0,\infty)$} if $\Gamma_2(u) \ge 0$ for all $u\in L(G,\mathbb{R})$. 

\begin{theorem}[\cite{LP18}, Theorem 1.2]
  \label{theorem 1.4}
Let $\mathcal{G} = (V,E,\omega)$ be an undirected,  weighted and connected graph.  Suppose that $\mathcal{G}$ is $\CD(0, \infty) $ and $0 = \lambda_1<  \lambda_2 \le  \cdots \le \lambda_n$ are eigenvalues of $-\Delta$.  Then for any natural number $ k \geq 2 $,  we have $\lambda_k  \leq C d_G k^2 \lambda_2$ where $C \coloneqq \left( \frac{20\sqrt{2}e}{e-1} \right)^2$.
\end{theorem}
\section{Positive definite functions}\label{sec:pd}
In this section,  we investigate real-valued functions whose Fourier coefficients are all non-negative.  These functions,  known as positive definite functions,  play a crucial role in harmonic analysis on groups \cite{Edwards79, Edwards82, Rudin90}.  To begin with, we summarize several basic properties of functions on finite abelian groups in the following lemma.
\begin{lemma}
\label{theorem 2.1}
Given a function $f:G\to \mathbb{C}$,  we denote by $H$ the subgroup of $\widehat{G}$ generated by $\supp(\widehat{f}) \coloneqq \{\chi \in \widehat{G}: \widehat{f} (\chi) \ne 0\}$. Then the followings hold:
  \begin{enumerate}[(a)]
    \item\label{theorem 2.1:item1} There is a well-defined function $\widetilde{f}:G/H^{\ann}\rightarrow \mathbb{C}$ such that $\widetilde{f} (x H^{\ann}) = f(x)$ for any $x \in G$.
 \item\label{theorem 2.1:item2}  Let $\pi : G\to G/H^\ann$ be the quotient map and let $ \pi^\ast:\widehat{G/H^{\ann}} \to \widehat{G}$ be the induced map  defined by $\pi^\ast (\rho) \coloneqq \rho \circ \pi$ is an isomorphism from $\widehat{G/H^{\ann}}$ to $H$.
 \item\label{theorem 2.1:item3} $\widetilde{f}=\sum_{\chi\in \supp{\widehat{f}}} \widehat{f}(\chi) (\pi^{*})^{-1}(\chi)$.
 \item \label{theorem 2.1:item4} Suppose that $f$ is real-valued.  For any $\chi \in \widehat{G}$,  we have $\widehat{f} (\chi^{-1}) = \overline{\widehat{f}(\chi)}$.  If moreover $\widehat{f}(\chi) \in \mathbb{R}$,  then $\widehat{f} (\chi^{-1}) = \widehat{f}(\chi)$.
  \end{enumerate}
\end{lemma}
\begin{proof}
To prove \eqref{theorem 2.1:item1},  it is sufficient to show $f(y)=f(x)$ for all $y \in x H^{\ann}$.  By definition,  there exists some $z \in H^{\ann}$ such that $y = x z$.  Since $H$ is generated by $\supp(\widehat{f})$,  $H^{\ann}=\{z \in G: \chi(z)=1, \; \chi \in \supp(\widehat{f})\}$.  This implies 
\[
f(y) =\sum_{\chi\in \supp(\widehat{f})}\widehat{f}(\chi)\chi(y) = \sum_{\chi\in \supp(\widehat{f})} \widehat{f}(\chi)\chi(x)\chi(z) = \sum_{\chi\in \supp(\widehat{f})}\widehat{f}(\chi)\chi(x) = f(x).
\] 
By Lemma \ref{Lemma 1.2},  \eqref{theorem 2.1:item2} is clear.  For \eqref{theorem 2.1:item3},  we notice that \eqref{theorem 2.1:item2} implies $\rho \coloneqq (\pi^{\ast})^{-1} (\chi) \in \widehat{G/H^\ann}$ for each $\chi \in \supp (\widehat{f}) \subseteq H \subseteq \widehat{G}$.  Moreover,  we have 
\[
\rho (x H^\ann) = \rho\circ \pi (x) = \pi^\ast (\rho) (x) = \chi (x),
\]
which leads to 
\[
\widetilde{f}(x H^\ann) = f(x) = \sum_{\chi \in \supp ( \widehat{f} )} \widehat{f}(\chi) \chi(x) = \sum_{\chi \in \supp ( \widehat{f} )} \widehat{f}(\chi) \left[ (\pi^\ast)^{-1}(\chi) \right] (x H^\ann).  
\]
By definition of $\widehat{f}(\chi)$ and the assumption that $f$ is real-valued,  it is straightforward to verify \eqref{theorem 2.1:item4}.
\end{proof}

For each $ f \in L(G) $,  we define 
\begin{equation}\label{eq:Rf}
R_f:  L(\widehat{G})\setminus \{0\} \to \mathbb{R},\quad R_f(v) =  \frac{|G| \langle \widehat{f} * v, v \rangle}{\langle v, v \rangle}.  
\end{equation}
\begin{lemma}
  \label{lemma 2.1} 
For each real-valued function on a finite abelian group $G$,  we have
$\max_{x \in G} f(x) = \max_{v \in L(\widehat{G}) \setminus \{0\}} R_f(v)$.
\end{lemma}
\begin{proof}
Suppose $\max_{x\in G} f(x) = f(x_0)$ for some $x_0 \in G$.  Denote $v \coloneqq |G| \widehat{\delta_{x_0}}$ where $\delta_{x_0}$ is the indicator function of $x_0$ defined by
\[
\delta_{x_0}(x) = 
\begin{cases}
1 \quad &\text{if $x = x_0$}, \\
0 \quad &\text{otherwise}.
\end{cases}
\]
Clearly,  we have $\widehat{v} = \delta_{x^{-1}_0}$. Since the Fourier transform is an isometry, we obtain 
  \[
  \frac{\langle \widehat{f} * v, v \rangle}{\langle v, v \rangle} = \frac{ \left\langle \widehat{\left( \widehat{f} * v \right)}, \widehat{v} \right\rangle}{\langle \widehat{v}, \widehat{v} \rangle} = \frac{\sum_{x \in G} f(x^{-1}) |\widehat{v}(x)|^2}{|G|\sum_{x \in G} |\widehat{v}(x)|^2} = \frac{f(x_0)}{ |G| } . \qedhere 
\]
\end{proof}
In the following,  we denote values of $f: G\to \mathbb{R}$ in the non-increasing order as
\[
\nu_1(f) \geq \nu_2(f) \ge  \cdots \geq \nu_n(f). \]
\begin{lemma}
  \label{theorem 2.2}
Let $ f $ be a real-valued function on a finite abelian group $ G $.  Given $v \in L(\widehat{G}) \setminus \{0\}$ such that $v(\chi) \ge 0$ for all $\chi \in \widehat{G}$,  we denote 
\[
\mu \coloneqq |\supp(v)|,\quad M \coloneqq \frac{|G| R_f(v) - \mu \sum_{i=1}^{k} \nu_i(f)}{|G| - \mu k}.
\]
If $ k \mu < |G| $ (resp.  $ k \mu \geq |G| $),  then $\nu_{k+1}(f) \geq M$ (resp.  $\nu_{k+1}(f) \le M$).
\end{lemma}
\begin{proof}
For each $ 1 \le i \le k $,  we choose $ x_i \in G $ such that $\nu_i(f) =  f(x_i) $.  Denote $ S \coloneqq \{x_1, \dots, x_k\} $.  We consider the function $f_1: G \to \mathbb{R}$ defined by 
    \[
    f_1(x) =
    \begin{cases}
        \nu_{k+1}(f), & \text{if } x \in S, \\
        f(x), & \text{otherwise}.
    \end{cases}
    \]
By definition,  we have $\nu_1(f_1) = \nu_{k+1}(f)$.  For each $\chi \in \widehat{G}$,  we have 
\[
    \widehat{f_1}(\chi) = \frac{1}{|G|}\sum_{x \in S} \nu_{k+1}(f) \overline{\chi(x)} + \frac{1}{|G|} \sum_{x \in G \setminus S} f(x) \overline{\chi(x)} = \widehat{f}(\chi) - \frac{1}{|G|} g(\chi),
    \]
where $ g : \widehat{G} \to \mathbb{C}$ is defined by 
    \[
    g(\chi) = \sum_{i=1}^{k} (\nu_i(f) - \nu_{k+1}(f)) \overline{\chi(x_i)}.
    \]   
Since $f_1$ is real-valued, Lemma~\ref{lemma 2.1} implies
\begin{equation}\label{theorem 2.2:eq1}
\nu_{k+1}(f) = \nu_1(f_1) \geq R_{f_1}(v) = R_f(v) - \frac{\langle g*v, v \rangle}{\langle v, v \rangle}.
\end{equation}
By assumption,  $ v$ is a non-negative function on $\widehat{G}$,  which leads to 
\begin{align}
|g * v(\chi)| = \left| \frac{1}{|G|} \sum_{\rho \in \widehat{G}} g(\rho^{-1}) v(\chi \rho) \right| &= \left|\frac{1}{|G|} \sum_{\rho \in \widehat{G}} \sum_{i=1}^{k} (\nu_i(f) - \nu_{k+1}(f)) \overline{\rho^{-1}(x_i)} v(\chi \rho) \right|  \label{theorem 2.2:eq2} \\
                  &\leq \frac{1}{|G|} \sum_{\rho \in \widehat{G}} \sum_{i=1}^{k} (\nu_i(f) - \nu_{k+1}(f)) v(\chi \rho) \nonumber \\ 
                  &= \left( \sum_{i=1}^{k} (\nu_i(f) - \nu_{k+1}(f)) \right) \langle v, 1 \rangle.  \nonumber
\end{align}
We notice that $\langle v, 1 \rangle = \langle v,  \delta_{\supp(v)} \rangle$ where $\delta_{\supp(v)} : \widehat{G} \to \mathbb{R}$ is the indicator function of $\supp(v)$ defined by 
\[
\delta_{\supp(v)}(\chi) = \begin{cases}
1  &\text{~if~} \chi \in \supp(v), \\
0 &\text{~otherwise}.
\end{cases}
\]
Combining \eqref{theorem 2.2:eq1} with \eqref{theorem 2.2:eq2},  we obtain
    \[
    \begin{aligned}
      \nu_{k+1}(f) \geq& R_f(v) - \left( \sum_{i=1}^{k} (\nu_i(f) - \nu_{k+1}(f)) \right) \frac{|\langle v,  \delta_{\supp(v)} \rangle|^2}{\langle v, v \rangle}\\
                       \geq&R_f(v) - \left( \sum_{i=1}^{k} (\nu_i(f) - \nu_{k+1}(f)) \right)\left< \delta_{\supp(v)} , \delta_{\supp(v)} \right>\\
                       =&R_f(v) - \frac{\mu}{|G|} \left( \sum_{i=1}^{k} (\nu_i(f) - \nu_{k+1}(f)) \right).
    \end{aligned}
    \]
    Therefore, we may conclude that
    \[
    \left( 1 - \frac{k \mu}{|G|} \right) \nu_{k+1}(f) \geq R_f(v) - \frac{\mu}{|G|} \sum_{i=1}^{k} \nu_i(f)
    \]
and the desired inequality follows immediately.
\end{proof}

\begin{definition}[Positive definite function]
A function $f$ on $G$ is positive definite,  denoted as $f \succeq 0$,  if $\widehat{f}(\chi) \ge 0$ for any $ \chi \in \widehat{G} $.  
\end{definition}
In the literature,  positive definite functions are also called spectrally positive functions \cite{KM23}.  Given $f: G \to \mathbb{C}$,  we consider the linear operator
\begin{equation}\label{eq:AfG} 
M_f^G: L(\widehat{G}) \to L(\widehat{G}),\quad  M_f^G (v) = |G| \widehat{f} * v.
\end{equation}
In the context of linear algebra,  $R_f(v)$ defined in \eqref{eq:Rf} is the \emph{Rayleigh quotient} of $v$ with respect to the linear operator $M^G_f$.  According to the celebrated Bochner's Theorem \cite{Bochner33,Weil41,Rudin90},  $f \succeq 0$ if and only if for any $M_{\widehat{f}}^{\widehat{G}}$ is a Hermitian positive semidefinite linear operator.

Since $f(x) = \sum_{\chi \in \widehat{G}} \widehat{f}(\chi) \chi (x)$ and $|\chi(x)| = 1$,  we observe that for any real-valued positive definite function $f$,  it holds that $\nu_1(f) \coloneqq \max_{x \in G} f(x) = f(1) = \sum_{\chi \in \widehat{G}} \widehat{f}(\chi)$.  We denote $\argmax_{x\in G} f \coloneqq \{x \in G:  f(x)= \nu_1(f)\}$ and observe that $|\argmax_{x\in G} f|$ is determined by the support of $\widehat{f}$.  More precisely,  we have:
\begin{proposition}[Multiplicity of $\nu_1(f)$]
  \label{cor2.1}
Let $f: G \to \mathbb{R}$ be a real-valued positive definite function and let $H$ be the subgroup of $\widehat{G}$ generated by $\supp(\widehat{f})$.  Then $\nu_1(f)=\cdots =\nu_k(f)$ if and only if $|\widehat{G}/H|\ge k$.
\end{proposition}
\begin{proof}
Since $ f \succeq 0$,  $\max_{x \in G} f(x) = \sum_{\chi \in \supp(\widehat{f})} \widehat{f}(\chi)$.  We have $H^\ann = \argmax_{x \in G} f$.  Therefore,  we conclude that $\nu_1(f)=\cdots =\nu_k(f)$ if and only if $|H^{\ann}|\ge k$.  According to Lemma \ref{Lemma 1.1},  we obtain $|H^{\ann}| = |\widehat{G}/H|$ and this completes the proof.
\end{proof}
We have the following observation which is a direct consequence of Proposition~\ref{cor2.1}.
\begin{corollary}\label{cor:multiplicity of nu1}
If $f\succeq 0$ and $\supp(\widehat{f})$ does not generate $\widehat{G}$ then we must have $\nu_2(f) = \nu_1(f)$.
\end{corollary}

Applying Lemma~\ref{theorem 2.2} to a real-valued positive definite function with $k = 1$,  we obtain the proposition that follows.
\begin{proposition}
\label{cor 2.2}
Let $ f: G \to \mathbb{R}$ be a real-valued positive definite function and let $v \in L(\widehat{G}) \setminus \{0\} $ be a non-negative function.  If $\mu \coloneqq |\supp(v)| < |G|$ then  
  \[
  \nu_{2}(f) \geq \frac{|G| R_f(v) - \mu f(1)}{|G| - \mu }.
  \]
\end{proposition}
\section{A lower bound of $\nu_2(f)$}\label{sec:nu2}
In this section,  we investigate lower bounds of $\nu_2(f)$ for $f \succeq 0$. By Proposition~\ref{cor 2.2},  any non-negative function $v$ on $\widehat{G}$ with $|\supp(v)| < |G|$ provides a lower bound of $\nu_2(f)$ as long as we can bound $R_f(v)$ from below.  The ultimate goal of this section is to construct such $v$. 
\subsection{A generalized covering argument}\label{subsec:covering} We first present a framework that enables one to derive a lower bound of $R_f(v)$,  and consequently for $\nu_2(f)$.  This framework generalizes the existing covering argument within the context of graph theory \cite{Friedman93,joel,FT05}.  It is important to note that our approach is grounded in Fourier analysis, which sets it apart from the previous graph-theoretic method. 

Let $G_1$ be an abelian group which is not necessarily finite.  We denote by $L_0(G_1)$ the space of finitely supported complex-valued functions on $G$.  Suppose that there is a group homomorphism $\eta: G_1\to  \widehat{G}$.  This induces a homomorphism 
\begin{equation}\label{eq:eta*}
\eta_{\ast}: L_0(G_1) \to L(\widehat{G}),\quad    \eta_{{*}}(h)(\chi) =  \begin{cases}
      \sum_{z \in \eta^{-1}(\chi)} h(z) & \text{if } \eta^{-1}(\chi)\neq \emptyset, \\
      0 & \text{otherwise}.
    \end{cases}
\end{equation}
Since $\supp(h)$ is finite,  the summation in \eqref{eq:eta*} is finite and $\eta_{\ast}$ is well-defined.  

Furthermore,  we assume that there is a subset $S$ of $G_1$ such that $\eta(S)= \supp(\widehat{f})$ and $|S|=|\supp(\widehat{f})|$.  In particular,  $\eta|_{S}$ is a bijection between $S$ and $\supp(\widehat{f})$. Thus,  there is a well-defined homomorphism
\begin{equation}\label{eq:AfG1}
M_f^{G_1}:L_0(G_1)\to L_0(G_1),\quad 
\left[ M_f^{G_1}(h) \right] (z) =  \sum_{w \in S}\widehat{f}( \eta(w)^{-1} ) h(w z).
\end{equation}
If $G_1 = \widehat{G}$ and $\eta$ is the identity map,  then $M_f^{G_1}$ defined by \eqref{eq:AfG1} coincides with $M_f^G$ in \eqref{eq:AfG}.
\begin{lemma}
\label{lemma 2.2}
Let $G$ be a finite abelian group and let $f \in L(G)$.  Suppose that $G_1$ is an abelian group and $\eta: G_1 \to \widehat{G}$ is a group homomorphism such that $\eta(S) = \supp(\widehat{f})$ for some $S \subseteq G_1$ of cardinality $|S| = |\supp (\widehat{f})|$ .  Then the following diagram commutes.  
  \[
  \begin{tikzcd}
  L_0(G_1)\arrow[r,"M^{G_1}_f"] \arrow[d,"\eta_{\ast}" '] & L_0(G_1) \arrow[d,"\eta_{\ast}"] \\
  L(\widehat{G}) \arrow[r,"M^{G}_{f}"'] & L(\widehat{G})
  \end{tikzcd}
  \]
Here $M^{G}_{f}$,  $\eta_{\ast}$ and $M_f^{G_1}$ are respectively defined by \eqref{eq:AfG},  \eqref{eq:eta*} and \eqref{eq:AfG1}.  
\end{lemma}
\begin{proof}
For each $h\in L_0(G_1)$ and $\chi \in \widehat{G}$,  we denote $v \coloneqq \eta_{\ast}(h)$ and $v_1 \coloneqq M^{G_1}_{f}(h)$.  On the one hand,  we have
\[
\left[ (M^{G}_{f} \circ\eta_{*} ) (h) \right] (\chi)
= \left[ M^{G}_f(v) \right] (\chi)
=\sum_{\rho \in \supp(\widehat{f})}\widehat{f}(\rho^{-1})v(\rho\chi)
=\sum_{\rho\in \supp(\widehat{f})}\sum_{z \in\eta^{-1}(\rho \chi)}\widehat{f}(\rho^{-1})h(z).
\]
On the other hand,  we have $\left[(\eta_{*}\circ M^{G_1}_f)(h)\right](\chi) = \left[ \eta_{\ast}(v_1) \right](\chi) =\sum_{z \in\eta^{-1}(\chi)}v_1(z)$ and 
\begin{align*}
\sum_{z \in\eta^{-1}(\chi)}v_1(z)
                                     &=\sum_{z \in\eta^{-1}(\chi)}\sum_{w \in S}\widehat{f}(\eta(w)^{-1})h(w z) \\
                                     & =\sum_{w \in S}\sum_{z\in \eta^{-1}(\chi)}\widehat{f}(\eta(w)^{-1})h(w z) \\
                                     &=\sum_{w \in S}\sum_{z \in\eta^{-1}(\eta(w)\chi)}\widehat{f}(\eta(w)^{-1})h(z)\\
                                     &=\sum_{\rho \in \supp(\widehat{f})}\sum_{z\in\eta^{-1}(\rho\chi)}\widehat{f}(\rho^{-1}) h(z).  \qedhere
\end{align*}
\end{proof}

Let $R_f:  L(\widehat{G})\setminus \{0\} \to \mathbb{R}$ be the function defined by \eqref{eq:Rf}.  By definition,  for each $v \in L(\widehat{G})\setminus \{0\}$,  $R_f(v)$ is the Rayleigh quotient of $v$ with respect to $M_f^G$.  Thus,  as a direct consequence of Lemma~\ref{lemma 2.2},  we obtain a lower bound of $R_f(v)$.
\begin{corollary}[Lower bound of Rayleigh quotient]
\label{theorem 2.3}
Suppose that $f \in L(G)$ has real Fourier coefficients and $C$ is a real number.  If $h: G_1 \to \mathbb{R}$ is a finitely supported non-negative function such that $\left[ M_f^{G_1}(h) \right](z) \ge C h(z)$ for every $z\in G_1$.  Then $R_{f}(\eta_{*}(h))\ge C $. 
\end{corollary} 
\begin{proof}
Let $v \coloneqq \eta_{*}(h)$.  By Lemma~\ref{lemma 2.2}, we have $M^{G}_f (v) = \left[ \eta_{*}\circ M_f^{G_1} \right] (h)\ge C \eta_{*}(h) = C v$.  This together with the non-negativity of $v$ implies
\[
R_{f}(\eta_{*}(h)) = R_{f}(v)= \frac{\left<M_f^G(v), v \right>}{\lVert v \rVert^2}  \ge C.   \qedhere
\]
\end{proof}

\begin{proposition}\label{prop:general bound for nu2}
Suppose that $f \in L(G)$ is real-valued and positive definite.  If $h: G_1 \to \mathbb{R}$ is a finitely supported non-negative function such that $\mu \coloneqq |\supp(\eta_\ast(h))| < |G|$ and $\left[ M_f^{G_1}(h) \right] \ge C h $ for some real number $C>0$.   Then 
\[
\nu_2(f) \ge \nu_1(f) +\frac{|G|  (C -\nu_1(f)) }{|G|-\mu }.
\]
\end{proposition}
\begin{proof}
Denote $v \coloneqq \eta_{\ast}(h)$.  According to Proposition~\ref{cor 2.2} and Corollary~\ref{theorem 2.3},  we obtain
\[
  \nu_2(f)\ge \frac{|G| R_{f}(v)- \nu_1(f) \mu }{|G|-\mu } = \nu_1(f) +\frac{|G|  (C -\nu_1(f)) }{|G|-\mu }.  \qedhere
\]
\end{proof}
\subsection{An estimate for $\nu_2(f)$}
If $ f $ is a real-valued function,  then 
\[
\supp(\widehat{f}) = \{ \chi_1,  \overline{\chi_1}, \dots, \chi_r, \overline{\chi_r},  \chi_{r+1}, \dots,  \chi_{r + t}\}.
\]
Here $\overline{\chi_i} = \chi_i$ if and only if $r +1 \le i \le r+t$.  
Denote $ G_1 \coloneqq \mathbb{Z}^r \times (\mathbb{Z}/ 2\mathbb{Z})^t $.  We consider a group homomorphism
\[
\eta : G_1 \to \widehat{G},\quad \eta(m_1,\dots,  m_{r+t}) = \prod_{i = 1}^{r+t} \chi_i^{m_i}.
\]
For each $1 \le i \le r+t$,  we denote by $e_i$ the element of $G_1$ whose entries are all zero except for the $i$-th,  which is $1$.  Let $S = \{ \pm e_1,  \dots,  \pm e_r,  e_{r+1},  e_{r+t}\}$ where .  Then clearly we have $|S| = |\supp(\widehat{f})|$ and $\eta(S) = \supp(\widehat{f})$.  
\begin{lemma}[Auxiliary function]
  \label{lemma 2.3}
Suppose that $f\in L(G)$ is a real-valued positive definite function such that $r \ge 1$  Let $G_1$ and $S$ be as above and let $m_1,\dots, m_r$ be fixed positive integers.  We define $h_0 \in L_0(G_1)$ by
\[
h_0(x_1,\dots, x_{r+t}) = \begin{cases}
\prod_{i=1}^r \sin\left( \frac{\pi x_i}{m_i} \right) \quad & \text{if~} 1\leq x_i\leq m_i-1,\; 1\leq i\leq r,\; r \ge 1 \\
          0 \quad &\text{otherwise}.
\end{cases}
\]
Then $\left[ M^{G_1}_f (h_0)  \right] (x)\geq \left(2 \sum_{i=1}^{r}\cos\left( \frac{\pi}{m_i} \right) \widehat{f}(\chi_i)  + \sum_{j=1}^t \widehat{f}(\chi_{r+j}) \right)h_0(x) $ for any $x\in G_1$. 
  \end{lemma}
  \begin{proof}
By definition and Lemma~\ref{theorem 2.1},  we have
\begin{align*}
\left[ M^{G_1}_f (h_0) \right] (x) &=\sum_{i=1}^{r}\left( \widehat{f}(\chi_i )h_0(e_i+x)+\widehat{f}(\overline{\chi_i})h_0(-e_i+x) \right) +\sum_{j=1}^{t}\widehat{f}( \chi_{r+j} ) h_0(e_{r+j}+ x)       \\
&=\sum_{i=1}^{r}  \widehat{f}(\chi_i ) \left( h_0(e_i+x)+ h_0(-e_i+x) \right)+\sum_{j=1}^{t}\widehat{f}( \chi_{r+j} ) h_0(e_{r+j}+ x)   \\
&=\sum_{i=1}^{r}  \widehat{f}(\chi_i ) \left(h_0(e_i+x)+ h_0(-e_i+x) \right)+\sum_{j=1}^{t}\widehat{f}( \chi_{r+j} ) h_0(x)  
\end{align*}
for each $x=(x_1,\cdots,x_{r+t}) \in G_1$.  We split the discussion into two cases: 
\begin{itemize}
\item There exists some $1\le i\le r$ such that $x_i<1$ or $x_i>m_i-1$.  Then we have $v(x) = 0$,  which implies 
\[
\left[ M^{G_1}_f (h_0) \right] (x) \geq 0 = \left( \sum_{i=1}^{r} 2 \cos \left( \frac{\pi}{m_i} \right) \widehat{f}(\chi_i) + \sum_{j=1}^t \widehat{f}(\chi_{r+j}) \right) h_0(x) 
\]
\item For each $1\leq i\leq r$,  $1 \leq x_i \leq m_i - 1$.  It is straightforward to verify that
\[
h_0(e_i + x) + h_0(-e_i +x) = 2 \cos\left(\frac{\pi}{m_i}\right) h_0(x).
\]
This leads to 
\[
\left[ M^{G_1}_f (h_0) \right] ( x )= \left( \sum_{i=1}^{r} 2 \cos \left( \frac{\pi}{m_i} \right) \widehat{f}(\chi_i) + \sum_{j=1}^t \widehat{f}(\chi_{r+j}) \right) h_0(x).  \qedhere
\]
\end{itemize}
\end{proof}

We recall that for any $h\in L_0(G_1)$,  $\chi \in \supp(\eta_{*}(h))$ if and only if $\eta^{-1}(\chi) \cap \supp(h) \neq \emptyset$,  which is further equivalent to $\chi \in \eta( \supp(h) )$.  Thus,  for the function $v$ defined in Lemma~\ref{lemma 2.3},  we have 
  \begin{equation}\label{eq:support estimate}
\mu \coloneqq |\supp(\eta_{*}(h_0))| = |\eta(\supp(v))| \leq |\supp(h_0)| = 2^t \prod_{i=1}^{r}(m_i - 1). 
  \end{equation}
\begin{proposition}
\label{theorem 2.4}
Let $f \in L(G)$ be a real-valued positive definite function such that $r \ge 1$.  Then 
\[
\nu_2(f) \ge \nu_1(f) - \frac{4 |G| \sum_{i=1}^{r} \widehat{f}(\chi_i) \sin^2\left( \frac{\pi}{2m_i} \right) }{|G|-2^t\prod_{i=1}^{r}(m_i-1)}
\]
for all positive integers $m_1,\dots,  m_r$ such that $2^t\prod_{i=1}^{r}(m_i-1)<|G|$. 
  \end{proposition}
  \begin{proof}
Denote $C \coloneqq 2 \sum_{i=1}^{r}\cos\left( \frac{\pi}{m_i} \right) \widehat{f}(\chi_i)  + \sum_{j=1}^t \widehat{f}(\chi_{r+j})$.  We recall that $\nu_1(f) = f(1) =2 \sum_{i=1}^{r} \widehat{f}(\chi_i) + \sum_{j=1}^t \widehat{f}(\chi_{r+ j})$.  According to \eqref{eq:support estimate},  Lemma~\ref{lemma 2.3} and Proposition~\ref{prop:general bound for nu2},  we obtain
\begin{align*}
 \nu_2(f) \ge \nu_1(f) +\frac{|G|  (C -\nu_1(f)) }{|G|-\mu } &= \nu_1(f)-\frac{2|G| \sum_{i=1}^{r} \widehat{f}(\chi_i)\left(1-\cos \left(\frac{\pi}{m_i} \right) \right)  }{|G|-\mu }\\
                    &\ge \nu_1(f)- \frac{4 |G| \sum_{i=1}^{r} \widehat{f}(\chi_i)\sin^2\left( \frac{\pi}{2m_i} \right) }{|G|-2^t\prod_{i=1}^{r}(m_i-1)}.   \qedhere
\end{align*} 
    \end{proof}
As a direct consequence of Proposition~\ref{theorem 2.4},  we have the following:
\begin{theorem}[Lower bound of $\nu_2(f)$]
\label{cor2.3}
Let $f \in L(G)$ be a real-valued positive definite function such that $r \ge 1$.  For each positive integer $m < \left(2^{-t} |G| \right)^{1/r}+1$,  it holds that
\begin{equation}\label{cor2.3:eq1}
\nu_2(f) \ge \nu_1(f) - \frac{4 |G| \sin^2\left( \frac{\pi}{2m} \right) }{|G|- 2^t (m-1)^r}  \sum_{i=1}^{r}\widehat{f}(\chi_i) \ge  \nu_1(f) \left( 1 - \frac{ \pi^2  |G|  }{2m^2 \left( |G|-2^t(m-1)^r \right)}  \right)
\end{equation}  
In particular,  given any positive integer $s$,  we have 
\begin{equation}\label{cor2.3:eq2}
\frac{\nu_1(f) - \nu_2(f)}{\nu_1(f)}  = O(|G|^{-4/s}).
\end{equation}  
for every real-valued positive definite function $f$ on a finite abelian group $G$ such that $|\supp(\widehat{f})| \le s$ and $r \ge 1$.
\end{theorem}
\begin{proof}
We observe that $2 \sum_{i=1}^{r} \widehat{f}(\chi_i) \le \nu_1(f)$.  Then \eqref{cor2.3:eq1} is obtained from Theorem~\ref{theorem 2.4} immediately by setting $m_1 = \cdots = m_r = m$.  For \eqref{cor2.3:eq2},  we notice that $ 2^t (m-1)^r \leq (m-1)^{t/2} (m-1)^r = (m-1)^{s/2}$ if $m \geq 5 $.  Moreover,  if $2^s < |G|$ and $m  = \alpha  |G|^{2/s}$ for some $\alpha \in (0,1)$,  then we have 
\begin{equation}\label{cor2.3:eq3}
\frac{\nu_1(f) - \nu_2(f)}{\nu_1(f)} \le \frac{2 \sin^2 \left(\frac{\pi}{2m} \right)}{ 1 - \alpha^{s/2}} \le  \frac{ \pi^2 }{2( 1 - \alpha^{s/2})m^2} = \frac{\pi^2 }{2 (1 - \alpha^{s/2}) \alpha^2} |G|^{-4/s}.  \qedhere
\end{equation}
\end{proof}
In the subsequent discussion,  applications of Theorem~\ref{cor2.3} are of particular interest to us.  Nonetheless,  Proposition~\ref{theorem 2.4} may provide a better lower bound of $\nu_2(f)$ than Theorem~\ref{cor2.3},  as the following example illustrates.
\begin{example}\label{ex:bound comparision}
On $G=\mathbb{Z}/{40 \mathbb{Z}}$,  we consider $f(x) = 2\chi(x)+2\chi^{-1}(x)+\chi^{5}(x)+\chi^{-5}(x)$ where $\chi(x) \coloneqq \exp( 2\pi ix/40)$.  A direct calculation reveals that $\nu_1(f) = 6$,  $\nu_2(f) =  \nu_3(f) = 5.364$ and $\nu_4(f) = \nu_{5}(f) = 3.804$.  It is straightforward to verify that the best lower bound of $\nu_2(f)$ obtained by Proposition~\ref{theorem 2.4} is $4.214$ when $(m_1,m_2) = (7,4)$,  while the bound obtained by Theorem~\ref{cor2.3} is $4.026$ when $m= 5$. 
\end{example}

We conclude this subsection by a brief discussion on  the best choice of $m$ in Theorem~\ref{cor2.3}.  To this end,  we need the following lemma.   
\begin{lemma}\label{lem:selectm}
Let $d$ be a positive integer and let $c$ be a positive real number.  We define a function $ \theta(x) = x^2 \left( 1 - c(x - 1)^d \right)$ on $[1,\infty)$.  Then $\theta$ has a maximizer and any maximizer $x_0$ of $\theta$ must satisfy  
        \[
    \left( \frac{2}{2c + cd} \right)^{\frac{1}{d}} < x_0 < \left( \frac{2}{2c + cd} \right)^{\frac{1}{d}} + 1.
    \]
\end{lemma}
\begin{proof}
Since $c,d >0$ and $\theta'(x) = 2x(1 - c(x-1)^d) - cdx^2 (x-1)^{d-1}  $,  we may conclude that $\theta'(1) = 2 > 0 $ and $\theta'(x) < 0 $ for sufficiently large $ x $.  This implies that $\theta$ attains its maximum at some point $x_0$ in $[1,\infty)$.  In particular,  $\theta'(x_0) = 0$.  Thus,  we obtain 
\[
(2c+cd)x_0^d >  2=2c(x_0-1)^d+cd(x_0-1)^{d-1} x_0  > (2c + cd) (x_0 - 1)^{d},
\] 
which implies $\left( 2/ (2c + cd) \right)^{1/d} < x_0 < \left( 2/ (2c + cd) \right)^{1/d} + 1$. 
\end{proof}
Now we apply Lemma~\ref{lem:selectm} to the lower bound in \eqref{cor2.3:eq1} with $c = 2^t/|G|$ and $d = r$.  
\begin{proposition}\label{prop:selectm}
Let $f, G, t,r,  \nu_1(f)$ and $\nu_2(f)$ be as in Theorem~\ref{cor2.3}.  If $\kappa \coloneqq \left( 2^{1-t}(2 +r)^{-1} |G| \right)^{1/r}$ is not an integer,  then the best lower bound given by \eqref{cor2.3:eq1} is 
\begin{equation*}
\nu_2(f) \ge  \nu_1(f) \left( 1 - \frac{ \pi^2 |G|  }{2{\lceil \kappa \rceil}^2 \left( |G|-2^t {\lfloor \kappa \rfloor}^r \right)}  \right).
\end{equation*}  
\end{proposition}
In Example~\ref{ex:bound comparision},  the best lower bound provided by Theorem~\ref{cor2.3} is achieved when $m= 5 = \left\lceil \sqrt{20} \right\rceil$.  This is obtained through a direct calculation,  as confirmed by Proposition~\ref{prop:selectm}.
\subsection{A digression from the perspective of graph theory}\label{subsec:digression}
Let $G$ be a finite abelian group.  For any function $w:G \to \mathbb{C}$ such that $w(x) = w(x^{-1})$ for any $x\in G$,  we denote by $\Cay(G,w)$ be \emph{weighted Cayley graph} \cite{L.Babai} determined by the pair $(G,w)$.  By definition,  $\Cay(G,w)$ is the weighted graph with vertex set $G$,  edge set $E = \{ \{x,y\} \subseteq G: x^{-1}y \in \supp(w) \}$ and the weight function $\omega:  V \times V \to \mathbb{C}$ defined by $\omega(x,  y) = w(x^{-1}y)$.  In particular,  $\Cay(G,   \delta_S)$ is the usual Cayley graph $\Cay(G,S)$ where $\delta_S$ is the indicator function of a symmetric subset $S\subseteq G$.  

We notice that there is a one to one correspondence between the set of weighted Cayley graphs and $L(\widehat{G})$.  Indeed,  we have 
\begin{equation}\label{eq:Phi}
\Phi: \{ \Cay(G,w):  w \in L(G) \} \to L(\widehat{G}),\quad  \Phi(  \Cay(G,w)) =|G| \widehat{w}.
\end{equation}
It is straightforward to verify that $\Phi^{-1}(f) = \Cay(G, \widehat{f})$ for any $f\in L(\widehat{G})$.  In particular,  $\Phi( \Cay(G, w) )$ is real-valued and positive definite if and only if $w(\chi) = w(\overline{\chi}) \ge 0$ for any $\chi \in \widehat{G}$. 

The existence of $\Phi$ enables us to view weighted Cayley graphs as functions on $G$.  Consequently,  we can investigate spectral properties of these graphs through their corresponding functions.  The \emph{adjacency matrix} of $\Cay(G,w)$ is the matrix $A(G,w) = (a_{x,y})_{x,y\in G} \in \mathbb{C}^{|G| \times |G|}$ such that $a_{x,y} = w(x^{-1} y)$. 
\begin{lemma}\cite[Corollary 3.2]{L.Babai}\label{proposition 3.1}
Suppose that $w(\chi) = w(\overline{\chi}) \ge 0$ for any $\chi \in \widehat{G}$.  Let $\mu_1 \ge \dots \ge \mu_{|G|}$ be eigenvalues of $A(G,w)$.  Then $\mu_i = \nu_i ( \Phi( \Cay(G, w) ) )$ for each $1 \le i \le |G|$.
\end{lemma}

\begin{proposition}\label{prop:weighted Cayley}
For any fixed positive integer $s$,  we have 
\[
\lambda_2( \Cay(G, w) ) = O(|G|^{-4/s}) 
\] 
for any $w\in L(G)$ such that $|\supp(w)| \le s$,  $\supp(w) \not\subseteq \{x \in G: x^2 = 1\}$ and $w(x) = w(x^{-1}) \ge 0$ for all $x \in G$.  Here $\lambda_2( \Cay(G, w) )$ is the second smallest eigenvalue of the minus normalized Laplacian operator of $ \Cay(G, w)$.
\end{proposition}
\begin{proof}
Let $f \coloneqq \Phi( \Cay(G, w) ) =  |G| \widehat{w} \in L(\widehat{G})$.  Then $\widehat{f}(x) = w(x^{-1}) = w(x)$.  According to Lemma~\ref{proposition 3.1} and Theorem~\ref{cor2.3},  we have $\lambda_2( \Cay(G, w)) = (\nu_1(f) - \nu_2(f))/\nu_1(f) =  O(|G|^{-4/s})$. 
\end{proof}
We recall that each $s$-regular Cayley graph on $G$ is of the form $\Cay(G,  \delta_S)$,  where $S \subseteq G$ is symmetric with $|S| = s$.  Proposition~\ref{prop:weighted Cayley} implies that $\lambda_2( \Cay(G,  S)) = s (1 - O(|G|^{-4/s}))$,  as long as $S \not\subseteq \{x \in G: x^2 = 1\}$.  This is also proved in \cite{joel} by the covering argument for graphs.
\section{Applications}\label{sec:app}
This section is concerned with three applications of  Theorem~\ref{cor2.3}: (a) Estimating values of arbitrary functions on finite abelian groups; (b) Obtaining lower bounds for the relaxation time and mixing time of random walks on finite abelian groups; (c) Deriving a lower bound for the size of the sumset of two subsets of finite abelian groups.
\subsection{Function value estimation}
Let $f$ be a real-valued function on $G$ with real Fourier coefficients.  We denote values of $f$ in the non-increasing order $\nu_1(f) \ge \cdots \ge \nu_{|G|}(f)$.  In this subsection,  we apply Theorem~\ref{cor2.3} to obtain lower and upper bounds of $\nu_k(f)$,  for each $3 \le k \le |G|$.

Given a real-valued positive definite function $f$ on $G$,  we consider the weighted Cayley graph $\Cay(\widehat{G},  \widehat{f})$ defined as in Subsection~\ref{subsec:digression}.  We recall from \cite{LY10} that $\Cay(\widehat{G},  \widehat{f})$ is a \emph{Ricci flat graph}.  As a consequence,  we have the lemma that follows.
\begin{lemma}\cite[Proposition~1.6]{LY10}
  \label{theorem 3.1}
If $f \in L(G)$ is real-valued and positive definite,  then $\Cay(\widehat{G},\widehat{f})$ is $\CD(0,\infty)$.  
\end{lemma}

For a real-valued $f \in L(G)$,  we define $t \coloneqq \lvert \{ \chi \in \supp(\widehat{f}): \chi^2 = 1 \} \rvert$ and $r \coloneqq (|G| - t)/2$.  The proof of the following lemma is inspired by an spectral estimate for Cayley graphs \cite[Corollary~4.7]{LP18}.
\begin{lemma}\label{prop:bound nuk}
Let $C \coloneqq \left( \frac{20\sqrt{2}e}{e-1}\right)^2$ and let $f \in L(G)$ be a real-valued positive definite function with $r \ge 1$.  Then for any positive integer $m < \left(2^{-t} |G| \right)^{1/r}+1$,  we have 
\[
\nu_k(f) \ge \nu_1(f) - C k^2 \frac{\pi^2 |G| }{2m^2\left( |G| - 2^t (m-1)^r \right)} \nu_1(f)^2.
\] 
\end{lemma}
\begin{proof}
Let $\lambda_1\ge \cdots \ge \lambda_{|G|}$ be eigenvalues of the adjacency matrix $A(\widehat{G},  \widehat{f})$ of $\Cay(\widehat{G},  \widehat{f})$.  According to Lemma~\ref{proposition 3.1},  we have $\lambda_k  = \nu_k (f)$ for any $1 \le k \le |G|$.  By Theorem~\ref{theorem 1.4} and Lemma~\ref{theorem 3.1},  we obtain $ \nu_1(f) -\nu_k(f)\le Ck^2( \nu_1(f) - \nu_2(f)) \nu_1(f)$.  The proof is complete by Theorem~\ref{cor2.3}.
\end{proof}

We observe that for each real-valued function $f$ with real Fourier coefficients,  there exist real-valued positive definite functions $f_1$ and $f_2$ such that $f = f_1 - f_2$.  
\begin{lemma}\label{lem:general nuk}
Let $ f_1 $ and $ f_2 $ be real-valued positive definite functions on $G$ and let $f \coloneqq  f_1 - f_2 $.  For any $1 \le k \le |G|$,  we have 
  \[
  \nu_k(f_1) - \nu_1(f_2) \leq \nu_k(f) \leq \nu_1(f_1) - \nu_{n-k+1}(f_2).
  \]
\end{lemma}
\begin{proof}
Assume $ \nu_i(f_1) = f_1(x_i)$ for some $x_i \in G$ where $1 \le i \le k$.  Then
  \[
  f(x_i) = f_1(x_i) - f_2(x_i) \geq \nu_k(f_1) - \nu_1(f_2),
  \]
which implies $ \nu_k(f) \geq \nu_k(f_1) - \nu_1(f_2)$.  For the upper bound,  we consider $ -f = f_2 - f_1 $.  By the above argument,  we obtain
    \[
 -\nu_k(f) =   \nu_{n-k+1}(-f)  \ge \nu_{n-k+1}(f_2) - \nu_1(f_1). \qedhere
\]
\end{proof}
For $i = 1,2$,  we denote 
\[t_i  \coloneqq \lvert \{ \chi \in \supp(\widehat{f_i}): \chi^2 = 1 \} \rvert,\quad r_i \coloneqq (|G| - t_i)/2. 
\]
Lemma~\ref{lem:general nuk} together with Proposition~\ref{prop:bound nuk} leads to the following estimate of $\nu_k(f)$.
\begin{proposition}\label{prop:function value}[Lower and upper bound for $\nu_k(f)$]
Let $C \coloneqq \left( \frac{20\sqrt{2}e}{e-1}\right)^2$ and let $f$,  $f_1$,  $f_2$ be as in Lemma~\ref{lem:general nuk}.  Suppose $k \ge 3$.  If $r_1,  r_2 \ge 1$,  then for any positive integers $ m_i < (2^{-t_i}|G|)^{1/r_i}+1$,  $i = 1,2$,  we have 
\[
- \frac{Ck^2  \pi^2 |G| \nu_1(f_1)^2  }{2m_1^2 \left( |G| - 2^{t_1} (m_1-1)^{r_1} \right)} 
\leq \nu_k(f) - f(1)  
\leq  \frac{C (n-k+1)^2 \pi^2 |G| \nu_1(f_2)^2}{2m_2^2 \left( |G| - 2^{t_2} (m_2-1)^{r_2} \right)} .
\] 
\end{proposition}
\subsection{Relaxation and mixing time estimation}
The second application of Theorem~\ref{cor2.3} is concerned with the relaxation and mixing times of random walks on finite abelian groups.  To begin with,  we briefly recall some necessary notions,  standard references for which are \cite{Saloff04,LPW09}.  A \emph{probability distribution} on a finite abelian group $G $ is a function $p: G \to [0, 1] $ such that $\sum_{g \in G} p(g) = 1$.  The \emph{random walk} on $G$ driven by $p$ is the \emph{Markov chain} with the \emph{state space} $G$ and \emph{transition probability matrix} $K = (K_{xy})_{x,y\in G}$ where $K_{xy} \coloneqq p(x^{-1}y)$.  We further assume that $p$ satisfies the following two conditions: 
\begin{enumerate}[(a)]
\item For each $x \in G$,  $p(x) = p(x^{-1})$.   \label{eq:p-condition1}
\item For each $x \in G$,  $\{xy:  y \in \supp(p)\}$ generates $G$.  \label{eq:p-condition2}
\end{enumerate}
The random walk driven by such $p$ is reversible \cite[Section~2.2]{Saloff04}, irreducible and aperiodic \cite[Proposition~2.3]{Saloff04}.

 Let $1 = \lambda_1  \ge \lambda_2 \ge \cdots \ge \lambda_{|G|} \ge -1$ be the eigenvalues of the transition probability matrix $K \in \mathbb{R}^{|G|\times |G|}$.  The \emph{relaxation time} $T_{\text{rel}}$ and the \emph{absolute relaxation time} $T^\star_{\text{rel}}$ of the random walk driven by $p$ are respectively defined by
\[
T_{\text{rel}} \coloneqq \frac{1}{1 - \lambda_2}, \quad T^\star_{\text{rel}} \coloneqq  \frac{1}{ 1 - \max\{ |\lambda_2|, |\lambda_{|G|}| \}}.
\]
\begin{proposition}[Lower bound of relaxation time]
  \label{relaxation time}
For any positive integer $s$ and probability $p$ on $G$ satisfying \eqref{eq:p-condition1},  \eqref{eq:p-condition2} and $|\supp(p)| \le s$,  we have $ T^\star_{\text{rel}}  = \Omega(|G|^{4/s})$ and $T_{\text{rel}} = \Omega(|G|^{4/s})$.
\end{proposition}
\begin{proof}
By definition,   we have $T^\star_{\text{rel}}  \ge T_{\text{rel}}$.  Thus,  it suffices to prove $T_{\text{rel}} = \Omega(|G|^{4/s})$.  If \( \supp(p) \not\subseteq \{ x \in G : x^2 = 1 \} \), we define \( f \coloneqq |G| \widehat{p} \) so that \( \widehat{f}(x) = p(x^{-1}) \) for any $x\in G$. Hence we have  $\lambda_2 = \nu_2(f) \leq \nu_1(f) = 1$.  According to Theorem~\ref{cor2.3},  we obtain $1 - \lambda_2 = O(|G|^{-4/s})$ and the lower bound of $T_{\text{rel}} $ follows immediately.  It is left to consider the case where \( \supp(p) \subseteq \{ x \in G : x^2 = 1 \} \).  By assumption,  $\supp(p)$ generates $G$.  Thus,  we must have $|G| \leq \sum_{k=0}^{s} \binom{s}{k} = 2^s$,  which implies $|G|^{4/s} \le 16$ and $T_{\text{rel}} = \Omega(|G|^{4/s})$ holds trivially.  
\end{proof}
\begin{remark}
In earlier works such as \cite{Greenhalgh89} and  \cite{Hough17},  it was proved that $T_{\text{rel}} = \Omega (p^{2/s})$ for $G = \mathbb{Z}/p\mathbb{Z}$.  Very recently,  this lower bound is extended to any finite abelian group \cite[Theorem~6.1]{MR23}.  According to Proposition~\ref{relaxation time},  the exponent of the existing lower bound can be further improved from $2/s$ to $4/s$.
\end{remark}
The exponent of the lower bound in Proposition~\ref{relaxation time} is tight,  as demonstrated by the following example.
\begin{example}
Let \(G = \text{Cay}(\mathbb{Z}/n\mathbb{Z}, \{1, -1\})\) and let \(p\) be the probability on $G$ defined by
  \[
  p(x) = 
  \begin{cases}
  \frac{1}{2} & \text{if } x \in \{1, -1\} \\
  0 & \text{otherwise}.
  \end{cases}
  \]
In this case,  we have $|G| = n$ and $s = 2$ and Proposition~\ref{relaxation time} yields $T_{\text{rel}} = \Omega(n^2)$ and  $T^\star_{\text{rel}} = \Omega(n^2)$.  On the other hand,  a direct calculation \cite[Example 12.3.1]{LPW09} indicates that $T_{\text{rel}}$ is of order $n^2$.  Moreover,  when $n$ is odd,  $T^\star_{\text{rel}}$ is also of order $n^2$.
\end{example}

Let \( G_1, G_2, \dots, G_d \) be finite abelian groups and let $p_1,\dots,  p_d$ be probability distributions on these groups,  respectively.  Suppose that $w$ is a probability distribution on $\{1,\cdots,d\}$.  Denote $G \coloneqq \prod_{j=1}^d G_j$ and define the matrix $K = (K_{\mathbf{x}, \mathbf{y}}) \in \mathbb{R}^{|G| \times |G|}$ by 
\[
K_{\mathbf{x}, \mathbf{y}} \coloneqq \sum_{j=1}^{d} \left( \prod_{i \neq j} \delta_{x_i,y_i} \right) w(j) p_j(x_j^{-1} y_j), 
\]
where $\mathbf{x} = (x_1,\dots,  x_d),  \mathbf{y} = (y_1,\dots,  y_d)$ are elements of $G$ and 
\[
\delta_{x,y} = \begin{cases}
1 \quad &\text{if~}x = y \\
0 \quad &\text{otherwise}.
\end{cases}
\]
The random walk on $G$ with the transition probability matrix $K$ is called the \emph{product chain} of random walks on $G_1,\dots,  G_d$ driven by $p_1,\dots,  p_d$,  respectively.
\begin{corollary}[Lower bound of relaxation time of product chain]\label{cor:pc}
Let $s$ be a fixed positive integer.  If $p_1,\dots,  p_d$ satisfy conditions~\eqref{eq:p-condition1} and \eqref{eq:p-condition2} with $\max \{ |\supp(p_j)|:1 \le j \le d,\; w(j) > 0 \} \le s$,  then we have $T^\star_{\text{rel}} =  \Omega( w(j)^{-1} |G_j|^{4/s})$ and $T_{\text{rel}} =  \Omega( w(j)^{-1} |G_j|^{4/s})$,  for any $1 \le j \le d$ such that $w(j) > 0$.
\end{corollary}
\begin{proof}
Since $T^\star_{\text{rel}} \ge T_{\text{rel}}$,  it is sufficient to establish the lower bound of $T_{\text{rel}}$.  For each $1 \le j \le d$,  we denote by $K_j$ the transition probability matrix of the random walk on $G_j$ driven by $p_j$.  According to \cite[Lemma~12.11]{LPW09},  we have 
\[
T_{\text{rel}}(G) = \max_{\substack{1 \le j \le d \\ w(j) > 0}} \{ w(j)^{-1} T_{\text{rel}}(G_j) \}.
\]
Here $T_{\text{rel}}(G)$ (resp. $T_{\text{rel}}(G_j)$) is the relaxation time of the product chain (resp.  random walk on $G_j$ driven by $p_j$).  This implies that $T_{\text{rel}}(G) \ge w(j)^{-1} T_{\text{rel}}(G_j) $ for any $1 \le j \le d$ such that $w(j) > 0$.  The desired lower bound of $T_{\text{rel}}$ follows from Proposition \ref{relaxation time}.
\end{proof}

Next,  we focus on estimating the mixing time, another essential characteristic of random walks. Given two measures $\mu$ and $\nu$ on $G$ and a positive integer $l$,  we respectively define 
\[
\lVert \mu  - \nu \rVert_{\text{TV}} \coloneqq \max_{A \subseteq G} \lvert \mu(A) - \nu(A) \rvert,\quad 
p^{(l)} \coloneqq \underbrace{p \ast \cdots \ast p}_{l \text{~times}}.
\]
Let $u_0$ be the uniform probability on $G$ defined by $u_0(x) = 1/|G|$ for each $x \in G$.  For each $0 < \varepsilon <1/2$,  the \emph{mixing time} of the random walk on $G$ driven by $p$ is 
\[
T_{\mix}(\varepsilon) \coloneqq \min \{m: \lVert p^{(m)}  - u_0 \rVert_{\text{TV}}  \le \varepsilon \}.
\]
\begin{proposition}[Lower bound of mixing time ]
\label{mixing time}
For any positive integer $s$ and probability $p$ on $G$ satisfying \eqref{eq:p-condition1},  \eqref{eq:p-condition2} and $|\supp(p)| \le s$,  we have 
\[
\frac{T_{\mix}(\varepsilon)}{\log \left( \frac{1}{2 \varepsilon} \right)} = \Omega(|G|^{4/s}),
\]
\end{proposition}
\begin{proof}
By \cite[Theorem~12.4]{LPW09},  we have 
\begin{equation}\label{mixing time:eq1}
T_{\mix}(\varepsilon)  \ge \left( T^{\star}_{\text{rel}} - 1 \right) \log \left( \frac{1}{2\varepsilon} \right),
\end{equation}
and the proof is complete by Proposition~\ref{relaxation time}.
\end{proof}

 Let $s$ be a positive integer.  Suppose that $G$ be a finite abelian group and $p$ is a probability distribution on $G$ satisfying \eqref{eq:p-condition1} and \eqref{eq:p-condition2} with $|\supp(p)| \le s$.  For $0 < \varepsilon < 1/2$,  we denote by $T_{\text{mix}}(\varepsilon)$ the mixing time of the product chain of random walks on $G_1 = \cdots = G_d = G$ driven by $p_1 = \cdots = p_d = p$ with $w: \{1,\dots,  d\} \to [0,1]$ defined by $w(j) \coloneqq 1/d$,  $1 \le j \le d$.  We may derive two different types of lower bounds for $T_{\text{mix}}(\varepsilon)$.
\begin{corollary}[Lower bound of mixing time for product chain]
For any fixed integer $s$,  we have $T_{\text{mix}}(\varepsilon) = \Omega(d|G|^{4/s}) \log \left( \frac{1}{2\varepsilon} \right)$ and $T_{\text{mix}}(\varepsilon) \ge  \Omega( d \log d|G|^{4/s}) + O(d) \log\left( \frac{1}{\varepsilon} - 1 \right)$.
\end{corollary}
\begin{proof}
According to Corollary~\ref{cor:pc} and \eqref{mixing time:eq1},  we have $T_{\text{mix}}(\varepsilon) = \Omega(d|G|^{4/s}) \log \left( 1/(2\varepsilon) \right)$.
On the other hand,  by \cite[Example~13.10]{LPW09} and Proposition~\ref{relaxation time},  we obtain
\[
T_{\text{mix}}(\varepsilon) \ge  \Omega( d \log d|G|^{4/s}) + O(d) \log\left( \frac{1}{\varepsilon} - 1 \right).  \qedhere
\]
\end{proof}
\subsection{Sumset size estimation}
Our final application of Theorem~\ref{cor2.3} pertains to the size of the sumset of two subsets of finite abelian groups.  Let $A,  B$ be subsets of a finite abelian group $G$.  We denote
\[
A B \coloneqq \{x y: x\in A,  y \in B\},\quad A^{-1} \coloneqq \{x^{-1}: x\in A\}.
\]
The set $AB$ is called the \emph{sumset} of $A$ and $B$.  If $G$ is an additive group,  we denote the sumset of $A$ and $B$ by $A + B$.  We say that $A$ is \emph{symmetric} if $A = A^{-1}$.  The problem of estimating $|AB|$ lies at the core of additive combinatorics \cite{TV10}.  The goal of this subsection is to establish a lower bound of $|AB|$ by Theorem~\ref{cor2.3}.
\begin{lemma}\label{lem:AB}
Given positive definite functions $f,  g \in L(G)$,  we have $\supp(\widehat{fg}) = \supp(f) \supp(g)$.  If $g$ is also real-valued,  then $\supp(\widehat{fg}) = \supp(f) \supp(g)^{-1}$.  
\end{lemma}
\begin{proof}
We note that $\widehat{fg} = |G|\widehat{f} \ast \widehat{g}$.  Since both $f$ and $g$ are positive definite,  their Fourier coefficients are non-negative.  Thus,  for each $\chi \in \widehat{G}$,  $\chi \in \supp(\widehat{fg})$ if and only if there is some $\rho \in \widehat{G}$ such that $\widehat{f}(\chi \rho^{-1}) > 0$ and $\widehat{g}(\rho) > 0$.  This is further equivalent to $\chi \in \supp(f) \supp(g)$.  If $g$ is also real-valued,  Lemma~\ref{theorem 2.1} \eqref{theorem 2.1:item4} implies that $\supp(g) = \supp(g)^{-1}$.
\end{proof}

Suppose that $A,B$ are symmetric subsets of $G$ and $H$ is the subgroup of $G$ generated by $A \cup B$. Let $f,  g$ be real-valued positive definite functions on $\widehat{G}$ such that $\supp(\widehat{f}) = A$ and $\supp(\widehat{g}) = B$.  By Lemma~\ref{lem:AB},  we have $\supp(\widehat{fg}) = AB \subseteq H$.  Suppose that $h \coloneqq  \widetilde{fg}: \widehat{G}/H^{\ann}  \to \mathbb{R}$ is the function induced by $fg$ in Lemma~\ref{theorem 2.1} \eqref{theorem 2.1:item1} and $\pi: \widehat{G} \to \widehat{G}/H^{\ann} $ is the quotient map.  Proposition~\ref{cor2.1} indicates that 
\[
\nu_2( h ) = \mu_2(f,g) \coloneqq \max \{ f( \chi) g(\chi): \chi \in \widehat{G},\; f(\chi) g(\chi )< \nu_1(f) \nu_1(g)\}.
\]
\begin{proposition}[Lower bound of sumset size]\label{lem:sumset}
If $|A| |B| \le 2 \log |H|/ \log (m-1)$ for some integer $m \ge 5$,  then we have 
\[
2 \frac{ \log |H|}{ \log (m-1)} \ge |AB| \ge  2\frac{\log |H| + \log \left( 1 - \frac{\pi^2}{2 m^2}\left( 1 - \frac{ \mu_2(f,g) }{\nu_1(f)\nu_1(g)}\right)^{-1} \right) }{\log (m-1)}.
\]
\end{proposition}
\begin{proof}
The first inequality is obvious since $|AB| \le |A| |B|$. According to Lemma~\ref{theorem 2.1} \eqref{theorem 2.1:item3},  we may derive that  
\begin{align*}
h  &= \sum_{z \in AB} \widehat{fg}(z) \left[ (\pi^\ast)^{-1}(z) \right] \\
&= \sum_{z \in AB} |G| (\widehat{f} \ast \widehat{g}) (z) \left[ (\pi^\ast)^{-1}(z) \right] \\
&= \sum_{z \in AB} \sum_{x \in A} \widehat{f}(x) \widehat{g}(x^{-1}z) \left[ (\pi^\ast)^{-1}(z) \right] \\
&= \sum_{x \in A} \widehat{f}(x) \sum_{y \in B} \widehat{g}(y) \left[ (\pi^\ast)^{-1}(xy) \right] \\
&= \sum_{x \in A} \widehat{f}(x) \left[ (\pi^{*})^{-1} (x) \right]  \sum_{y \in B} \widehat{g}(y) \left[ \pi^{*})^{-1}(y) \right].
\end{align*}   
Thus we conclude that $\nu_1 (h) = \nu_1(f) \nu_1(g)$.  Let $s \coloneqq | AB |$.  Then $|\supp(\widehat{h})| = |\supp(\widetilde{\widehat{fg}})|  =  s$. Since $m \ge 5$,  by \eqref{cor2.3:eq1},  we have  
\[
\nu_2(h) \ge  \nu_1(h) \left( 1 - \frac{\pi^2 |H| }{2m^2 \left( |H|- (m-1)^{s/2} \right)}  \right),
\]
from which the second inequality follows immediately.
\end{proof}
\begin{example}  
Let $p \ge 17$ be a prime and let \( G \coloneqq \mathbb{Z}/p\mathbb{Z} \times \mathbb{Z}/p\mathbb{Z}  \).  
We consider two subsets of $G$: $A = \{ \mathbf{v}_1,-\mathbf{v}_1, \mathbf{v}_2,-\mathbf{v}_2 \}$ and $B = \{\mathbf{v}_3, -\mathbf{v}_3\}$,  where $\{\mathbf{v}_1, \mathbf{v}_2\} $ is a basis of $ \mathbb{Z}/p \mathbb{Z} \times \mathbb{Z}/p \mathbb{Z} $ and $\mathbf{v}_3\in  \mathbb{Z}/p \mathbb{Z} \times \mathbb{Z}/p \mathbb{Z}$ is nonzero.  We define two functions $f,g:  \widehat{G} \to \mathbb{R}$ by
\[
f( \rho ) = \sum_{\mathbf{x} \in A}  \rho(\mathbf{x}),\quad g( \rho ) = \sum_{\mathbf{x} \in B}  \rho(\mathbf{x}).
\]
It is clear that both $f$ and $g$ are positive definite functions on $\widehat{G}$ and $\supp(\widehat{f}) = A$,  $\supp(\widehat{g}) = B$.  Indeed,  by identifying $\widehat{\widehat{G\,}}$ with $G$,  we have  
  \[
f(\mathbf{x}) = 2\left(\cos\left(\frac{2\pi \mathbf{x} \cdot \mathbf{v}_1}{p}\right) + \cos\left(\frac{2\pi \mathbf{x} \cdot \mathbf{v}_2}{p}\right)\right),\quad g( \mathbf{x} ) = 2\cos\left(\frac{2\pi \mathbf{x} \cdot \mathbf{v}_3}{p}\right),
  \]
where $\mathbf{x}\in G$ and $\mathbf{y} \cdot \mathbf{z}$ denotes the usual inner product of vectors $\mathbf{y},  \mathbf{z} \in G$.  According to Proposition~\ref{lem:sumset},  we obtain
\[
8 \geq |AB| \geq 2 \frac{2\log p + \log \left( 1 - \frac{\pi^2}{50} \left( 1 - \frac{\mu_2(f,g)}{8} \right)^{-1} \right)}{\log 4}.
\]   
In particular,  if $p = 17$,  $\mathbf{v}_1 = (1,0)$, $\mathbf{v}_2 = (9,9)$ and $\mathbf{v}_3 = (2,0)$,  then it is straightforward to verify that $\mu_2(f,g) = 5.71$,  which implies $|AB| \geq 7$.
\end{example}

We conclude this subsection by a brief discussion on the sumset size estimation for $\mathbb{Z}$.  Let $A$ and $B$ be finite symmetric subsets of $\mathbb{Z}$. Suppose that $n$ and $m$ are the maximum values of the elements in $A$ and $B$,  respectively.  We denote by $\overline{A},  \overline{B}$ images of $A,B$ in the quotient group $\mathbb{Z}/(n+m+1) \mathbb{Z}$.  It is obvious that $|\overline{A} + \overline{B}| = |A + B|$ and Proposition~\ref{lem:sumset} applies if $|A + B|$ is small.

Moreover,  we recall from \cite[Lemma 2.1 ]{TV10} that $\lvert A + A\rvert \le \binom{|A|+1}{2}$ for any finite subset $A \subseteq \mathbb{Z}$.   We notice that $\nu_1(\widehat{\delta_A}) = |A|$ and $\overline{A}$ generates $\mathbb{Z}_p$ for each prime $p \ge 2n + 1$.  Therefore,  as a consequence of Proposition~\ref{lem:sumset},  we may derive the corollary that follows. 
\begin{corollary}
  Let $A$ be a finite symmetric subset of $\mathbb{Z}$ and let $m\ge 5$ be an integer number.  For any prime number $p$ such that $p \ge 2 \max_{a\in A}\{a\} + 1$ and $2 \log p/\log (m-1) \ge \binom{|A|+1}{2}    $,  we have 
  \[
  |A + A| \ge  2\frac{\log p+ \log \left( 1 - \frac{\pi^2}{2 m^2}\left( 1 - \frac{ \nu_2 \left( \widehat{\delta_{\overline{A}}}^2 \right) }{|A|^2}\right)^{-1} \right) }{\log(m-1)}.
  \]
  \end{corollary}

\bibliographystyle{abbrv}
\bibliography{function_3}

\end{document}